\documentclass[reqno,12pt]{amsart}

\usepackage{color} 
\usepackage{ifpdf}
\ifpdf
    \usepackage[pdftex]{graphicx}
    \usepackage[pdftex]{hyperref}
    \hypersetup{
        unicode=false,          
        pdftoolbar=true,        
        pdfmenubar=true,        
        pdffitwindow=false,     
        pdfstartview={FitH},    
        pdftitle={MCP Article},      
        pdfauthor={Michael Holst},   
        pdfsubject={Mathematics},    
        pdfcreator={Michael Holst},  
        pdfproducer={Michael Holst}, 
        pdfkeywords={PDE, analysis, mathematical physics}, 
        pdfnewwindow=true,      
        colorlinks=true,        
        linkcolor=red,          
        citecolor=blue,         
        filecolor=magenta,      
        urlcolor=cyan           
    }

\else
    \usepackage{graphicx}
    \usepackage{pstricks}
    
    \newcommand{\href}[2]{#2}
\fi

\usepackage{times}
\usepackage{amsfonts}

\usepackage{amsmath}
\usepackage{amsthm}
\usepackage{amssymb}
\usepackage{amsbsy}
\usepackage{amscd}
\usepackage{stmaryrd}

\usepackage{enumerate}
\usepackage{verbatim}
\usepackage{subfigure}

\newtheorem{theorem}{Theorem}[section]
\newtheorem{corollary}[theorem]{Corollary}
\newtheorem{lemma}[theorem]{Lemma}

\newtheorem{assumption}[theorem]{Assumption}

\numberwithin{equation}{section}  

  \newcommand{\mnoter}[1]{}
  \newcommand{\mnoteb}[1]{}
  \newcommand{\mnoteg}[1]{}
  \newcommand{\mnotem}[1]{}

\definecolor{myblue}{rgb}{0.2,0.2,0.7}
\definecolor{mygreen}{rgb}{0,0.6,0}
\definecolor{mycyan}{rgb}{0,0.6,0.6}
\definecolor{myred}{rgb}{0.9,0.2,0.2}
\definecolor{mymagenta}{rgb}{0.9,0.2,0.9}
\definecolor{mywhite}{rgb}{1.0,1.0,1.0}
\definecolor{myblack}{rgb}{0.0,0.0,0.0}

\DeclareMathOperator{\osc}{osc}

\DeclareMathOperator{\tr}{tr}
\DeclareMathOperator{\supp}{supp}
\DeclareMathOperator{\tol}{tol}

\newcommand{\ab}[2]{\langle#1,#2\rangle}
\newcommand{\Th}{\mathcal{T}_h}
\newcommand{\Eh}{\mathcal{E}_h}

\newcommand{\lr}[1]{\llbracket#1\rrbracket}

\include{symbols}

\usepackage[all,tips]{xy}

\begin{document}

\title[Convergence and optimality of AMFEM in the FEEC Framework]
      {Convergence and Optimality of Adaptive Mixed Methods for Poisson's equation 
       in the FEEC Framework}     
  
\author[M. Holst, Y. Li, A. Mihalik, R. Szypowski ]{Michael Holst, Yuwen Li, Adam Mihalik, and Ryan Szypowski}  

\address{Department of Mathematics\\
         University of California San Diego\\ 
         La Jolla CA 92093}
\email{mholst@math.ucsd.edu, yul739@ucsd.edu, amihalik@math.ucsd.edu}

\address{Department of Mathematics and Statistics\\
         Cal Poly Pomona\\ 
         Pomona CA 91768}
\email{rsszypowski@csupomona.edu}

\thanks{MH was supported in part by NSF Awards~1620366, 1262982, and 1217175.}
\thanks{YL was supported in part by NSF Award~1620366.}
\thanks{AM and RS were supported in part by NSF Award~1217175.}

\date{\today}   

\begin{abstract}
Finite Element Exterior Calculus (FEEC) was developed by
Arnold, Falk, Winther and others over the last decade to exploit
the observation that mixed variational problems can be posed on a Hilbert complex, and Galerkin-type mixed methods can then be obtained 
by solving finite-dimensional subcomplex problems.
Chen, Holst, and Xu (Math.~Comp.~78 (2009) 35--53) established convergence and optimality of an 
adaptive mixed finite element method using Raviart--Thomas or Brezzi--Douglas--Marini elements for Poisson's equation 
on contractible domains in $\mathbb{R}^2$, which can be viewed as a boundary problem on the de Rham complex. Recently Demlow and Hirani (Found. Math. Comput.~14 (2014) 1337-1371) developed fundamental tools for a posteriori analysis on the de Rham complex.
In this paper, we use tools in FEEC to construct convergence and complexity results on domains with general topology and spatial dimension. In particular, we construct a reliable and efficient error estimator and a sharper quasi-orthogonality result using a novel technique. Without marking for data oscillation, our adaptive method is a contraction with respect to a total error incorporating the error estimator and data oscillation.

\end{abstract}

\maketitle


\tableofcontents

\clearpage


\section{Introduction}

An idea that has had a major influence on the
development of numerical methods for PDE applications
is that of \emph{mixed finite elements}, 
whose early success in areas such as computational
electromagnetics was later found to have surprising connections with
the calculus of exterior differential forms, including de~Rham
cohomology and Hodge
theory~\cite{Bossavit1988,Nedelec1980,Nedelec1986,GrKo2004}.
A core idea underlying these developments is
the {\em Helmholtz-Hodge} orthogonal decomposition 
of an arbitrary vector field $f \in(L^2(\Omega))^3$
into curl-free, divergence-free, and harmonic functions:
$$
f = \nabla p + \nabla \times q + h,
$$
where $p\in H_0^1(\Omega), q\in H(\text{curl},\Omega)$, and $h$ is harmonic (divergence- and curl-free).
The mixed formulation is explicitly computing the decomposition for $h=0$,
and finite element methods based on mixed formulations exploit this.
There is a connection between this decomposition and
{\em de Rham cohomology}; the space of harmonic forms is isomorphic 
to the first {\em de Rham cohomology} of the domain $\Omega$,
with the number of holes in $\Omega$ giving the first Betti number,
and creating obstacles to well-posed formulations
of elliptic problems.
A natural question is then: What is an appropriate mathematical framework 
for understanding this abstractly, that will allow for
a methodical construction of ``good'' finite element methods
for these types of problems?
The answer turns out to be the theory of {\em Hilbert Complexes}.
Hilbert complexes were originally studied in \cite{BrLe1992} 
as a way to
generalize certain properties of elliptic complexes, particularly the
Hodge decomposition and other aspects of Hodge theory.  
The \emph{Finite Element Exterior Calculus} (or {\em FEEC}) 
\cite{AFW06,AFW10} was developed to exploit this abstraction.
A key insight was that
from a functional-analytic point of view, a mixed variational
problem can be posed on a Hilbert complex: a differential
complex of Hilbert spaces, in the sense of \cite{BrLe1992}.
Galerkin-type mixed methods are then obtained by solving the
variational problem on a finite-dimensional subcomplex.
Stability and consistency of the resulting method, often shown using 
complex and case-specific arguments, are reduced by the framework to 
simply establishing existence of operators with certain properties that 
connect the Hilbert complex with its subcomplex, essentially giving a 
``recipe'' for the development of provably well-behaved methods.
\mnoter{
Red margin notes are urgent things.
}
\mnoteb{
Blue margin notes are just comments.
}

Due to the pioneering work of 
Babu\v{s}ka and Rheinboldt~\cite{Babuska.I;Rheinboldt.W1978}, 
adaptive finite element methods (AFEM) based on {\em a posteriori} error
estimators have
become standard tools in solving PDE problems arising in science 
and engineering
(cf.~\cite{Ainsworth.M;Oden.J2000,Verfurth.R1996,Repin.S2008}).
A standard adaptive algorithm has the general iterative structure:
\begin{equation}{\label{eq:adaptive}}
 \textsf{Solve} \longrightarrow \textsf{Estimate} \longrightarrow \textsf{Mark} \longrightarrow \textsf{Refine},
\end{equation}
where 
\textsf{Solve} computes the discrete solution 
       $u_\ell$ in a subspace $X_\ell\subset X$;
\textsf{Estimate} computes certain error estimators based on $u_\ell$, which
        are reliable and efficient in the sense that they are good
        approximation of the true error $u-u_\ell$ in the energy norm; 
\textsf{Mark} applies certain marking strategies based on the
        estimators; and finally, 
\textsf{Refine} divides each marked element and completes the mesh to
        to obtain a new partition, and subsequently
        an enriched subspace $X_{\ell+1}$.
The fundamental problem with the adaptive procedure~\eqref{eq:adaptive} 
is guaranteeing convergence of the solution sequence.
The first convergence result for~\eqref{eq:adaptive} was obtained by 
Babu\v{s}ka and Vogelius~\cite{Babuska.I;Vogelius.M1984} for 
linear elliptic problems in one space dimension.
The multi-dimensional case was open
until D\"orfler~\cite{Dorfler.W1996} proved convergence of~\eqref{eq:adaptive}
for Poisson's equation by using the so called D\"orfler marking, under the assumption that the initial mesh was fine 
enough to resolve the influence of data oscillation.
This result was improved by Morin, Nochetto, and 
Siebert~\cite{MS1}, in which the convergence was proved without conditions 
on the initial mesh, but requiring the so-called 
\emph{interior node property}, together with an additional marking step 
driven by data oscillation.
It was shown by Binev, Dahmen and 
DeVore~\cite{BBD04} for the first time that
AFEM for Poisson's equation in the plane has optimal computational complexity 
by using a special coarsening step.
This result was improved by Stevenson~\cite{RS07} by showing the optimal 
complexity in general spatial dimension without a coarsening step.
These error reduction and optimal complexity results were improved
in several aspects in~\cite{CKNS}.
In their analysis, the artificial assumptions of interior node and extra 
marking due to data oscillation were removed, and the convergence result 
is applicable to general linear elliptic equations.
The main ingredients of this new convergence analysis are the global upper
bound on the error given by the {\em a posteriori} estimator, 
orthogonality (or possibly only quasi-orthogonality) of the underlying
bilinear form arising from the linear problem, and a type of error indicator 
reduction produced by each step of AFEM. In another direction, Morin, Siebert, and Veeser~\cite{MSV} gave a plain convergence result of conforming AFEMs for a widge range of linear problems without using D\"orfler makring.
We refer to~\cite{NV} for a recent survey of 
convergence analysis of AFEM for linear elliptic PDE problems
which gives an overview of all of these results through 2012.
See also~\cite{HTZ08a} or an overview of various extensions
to nonlinear problems.

Of particular relevance here is the 2009 article
of Chen, Holst, and Xu~\cite{CHX}, where
convergence and optimality of an adaptive mixed finite element method (AMFEM) using Raviart--Thomas (RT) \cite{RT} or Brezzi--Douglas--Marini (BDM) \cite{BDM} elements
for Poisson's equation on simply connected polygons in $\mathbb{R}^2$ 
was established.
The main difficulty for convergence analysis of AMFEM is the lack of
minimization principle, and thus the failure of orthogonality. A main contribution of \cite{CHX} is
a quasi-orthogonality result on the error $\| \sigma - \sigma_h \|$. The proof is based on
the fact that the error is orthogonal to the divergence free 
subspace, while the part of the error that is not divergence free 
was bounded by the data oscillation using a discrete stability result.
We also mention that Becker and Mao \cite{BM} developed a convergent AMFEM with optimal comlexity using the lowest-order RT finite element in $\mathbb{R}^2$. They used a multigrid inexact solver in the \textsf{SOLVE} module, which is another direction of interest. 

In this paper, we generalize the results in~\cite{AA,CHX}
by analyzing the error $\| \sigma - \sigma_h \|$ in the FEEC framework, which allows us extend the convergence and complexity results for contractible
domains in two dimensions in~\cite{CHX} to domains of arbitrary 
topology and spatial dimension. In FEEC terminology, the methods considered in~\cite{CHX} are equivalent 
to those for solving the Hodge Laplacian problem when $k=n=2$. All of our results apply to the case $k=n$ for arbitrary $n\geq2$ and domains which are not necessarily contractible. Even in the case $k=n=2$, our quasi-orthogonality result is sharper than \cite{CHX} in the sense that it involves a local data oscillation. The quasi-orthogonality Theorem \ref{quasiortho} is motivated by Becker and Mao's result \cite{BM} in $\mathbb{R}^2$. With the sharper quasi-orthogonality, we are able to prove contraction of Algorithm \textsf{AMFEM} by defining the total error  $\|\sigma-\sigma_h\|^2+\rho\eta^2_{\Th}(\sigma_h,\Th)+\zeta\osc^2_{\Th}(f,\Th)$, see Theorem \ref{termination}. Comparing to \cite{CHX} using a separate marking driven by data oscillation, \textsf{AMFEM} uses a single marking step based on the estimator $\eta_{\Th}(\sigma_h,T)$. 

This paper is a revised version of the unpublished preprint \cite{HMS} in 2013. The reliability proof of $\eta_{\Th}$ remains basically unchanged but is stated in an algebraic way. The main improvement are two-fold. First we give a completely new proof of Theorem \ref{quasiortho}, a refined quasi-orthogonality result,   while \cite{HMS} follows the quasi-orthogonality proof in \cite{CHX}. Second, the contraction analysis of \textsf{AMFEM} is novel by using the aforementioned improved quasi-orthogonality and total error. In addition, several inaccuracies in \cite{HMS} such as proofs of Corollary \ref{dhdc} and quasi-optimality are fixed or removed.

Recently, there are several results on convergence and optimality of AMFEM in FEEC. Demlow \cite{Demlow} developed a convergent AFEM with optimal complexity for computing the space of harmonic forms. In \cite{CW1}, Chen and Wu developed a convergent AMFEM for solving the Hodge Laplacian with index $1\leq k\leq n-1$ with respect to the error $\|d(\sigma-\sigma_h)\|^2+\|d(u-u_h)\|^2$ on contractible domains. The second author \cite{YL} developed two AMFEMs for the Hodge Laplacian with index $1\leq k\leq n$ on Lipschitz domains with general topology. When $k=n$, his results can control and reduce the energy error $\|\sigma-\sigma_h\|_{H\Lambda^{n-1}}$ while \textsf{AMFEM} is dealing with the $L^2$ error $\|\sigma-\sigma_h\|$. Assuming sufficient regularity,  $\|\sigma-\sigma_h\|=O(h^{r+2})$ is of higher order than $\|\sigma-\sigma_h\|_{H\Lambda^{n-1}}=O(h^{r+1})$ when using the generalized BDM pair \eqref{BDM}. In addition, the quasi-orthogonality result Theorem \ref{quasiortho} is sharper than \cite{YL} and the proof is quite different.

The remainder of the paper is organized as follows.
In Section~\ref{sec:prelim} we introduce the notational and technical tools in FEEC 
needed for the paper.
In Section~\ref{sec:estimator} we present an error indicator with global reliability and local efficiency.
In Section~\ref{sec:quasiortho}, we construct the quasi-orthogonality result. The adaptive algorithm \textsf{AMFEM} is then presented in Section~\ref{sec:conv}, 
and we prove both convergence and optimality.

\section{Preliminaries}
\label{sec:prelim}

In this section we first review abstract Hilbert complexes.
We then examine the particular case of the de Rham complex.
We follow closely the notation and the general development 
of Arnold, Falk, and Winther in~\cite{AFW06,AFW10}.
We also discuss results from Demlow and Hirani in~\cite{DH}.
(See also~\cite{HoSt10a,HoSt10b} for a concise summary of Hilbert Complexes
in a yet more general setting.)
We then give an overview of the basics of Adaptive Finite Element Methods 
(AFEM), and the ingredients we will need to prove convergence and optimality
within the FEEC framework.
 
\subsection{Hilbert complexes }

We begin with a quick summary of some basic concepts and definitions.
A \emph{Hilbert complex} $(W, d )$ is a sequence of Hilbert spaces $W^k$ equipped with the inner product $\langle\cdot,\cdot\rangle$, closed and densely defined linear operators, $d^k$, which map their domain, $V^k \subset W^k$ to the kernel of $d^{k+1}$ in $W^{k+1}$.  A Hilbert complex is \emph{bounded} if each $d^k$ is a bounded linear map from $W^k$ to $W^{k+1}$  A Hilbert complex is \emph{closed} if the range of each $d^k$ is closed in $W^{k+1}$.
Given a Hilbert complex $(W, d)$, the subspaces  $V^k \subset W^k$ endowed with the graph inner product
\begin{equation*}
\langle u, v \rangle_{V} = \langle u, v \rangle + \langle d^k u, d^k v \rangle,
\end{equation*}
form a Hilbert complex $(V, d)$ known as the \emph{domain complex}.
By definition $d^{k+1}\circ$  $d^k = 0$, thus  $(V, d)$ is a bounded 
Hilbert complex.
Additionally, $(V, d)$ is closed if $(W, d)$ is closed.

The range of $d^{k-1}$ in $V^k$ will be represented by $\mathfrak{B}^k$, and the null space of $d^k$ will be represented by $\mathfrak{Z}^k.$  Clearly, $\mathfrak{B}^k \subset \mathfrak{Z}^k$.  The elements of $\mathfrak{Z}^k$ orthogonal to $\mathfrak{B}^k$ are the space of harmonic forms, represented by $\mathfrak{H}^k$.  For a closed Hilbert complex we can write the \emph{Hodge decomposition} of $W^k$ and $V^k$,
\begin{align}
W^k &= \mathfrak{B}^k \oplus \mathfrak{H}^k \oplus \mathfrak{Z}^{k \perp}, \\
V^k &= \mathfrak{B}^k \oplus \mathfrak{H}^k \oplus \mathfrak{Z}^{k \perp_V },
\end{align}
where $\perp$ denotes the orthogonal complement w.r.t.~$\langle\cdot,\cdot\rangle$ and $\mathfrak{Z}^{k\perp V}:=\mathfrak{Z}^{k\perp}\cap V^k$. We use $P_{\mathfrak{B}}, P_{\mathfrak{H}}, P_{\mathfrak{Z}^\perp}$ to denote the $L^2$ projections onto $\mathfrak{B}^{k}, \mathfrak{H}^{k}, \mathfrak{Z}^{k\perp}$ ,respectively. Another important Hilbert complex will be the \emph{dual complex} $(W, d^*)$, where $d^*_k$, which is an operator from $W^k$ to $W^{k-1}$, is the adjoint of $d^{k-1}$.  The domain of $d^*_k$ will be denoted by $V^*_k$. Let $\mathfrak{Z}_k^*$ denote the null space of $d_k^*$ and $\mathfrak{B}_k^*$ the range of $d_{k+1}^*$.
For closed Hilbert complexes, an important result will be the \emph{Poincar\'e inequality},
\begin{equation}\label{poincare}
\|v\|_V \le c_P\|d^kv\|_W, \hspace{.2cm} v\in \mathfrak{Z}^{k\perp}.
\end{equation}
In addition, we have the important relation $\mathfrak{Z}_k^\perp=\mathfrak{B}_k^*$. The de Rham complex is the practical complex where general results we show on an abstract Hilbert complex will be applied.  

\subsubsection*{The abstract Hodge Laplacian}

Given a Hilbert complex $(W,d)$, the operator $L = dd^* + d^*d$,  $W^k \rightarrow W^{k}$ will be referred to as the \emph{abstract Hodge Laplacian}.  For $f\in W^k$, the Hodge Laplacian problem can be formulated as the problem of finding $u \in W^k$ such that
\begin{equation*}
\langle du, dv \rangle +\langle d^* u, d^* v \rangle = \langle f, v \rangle,\quad v \in V^k \cap V^*_k.
\end{equation*}

A necessary condition for the solution to exsit is $f\perp\mathfrak{H}^{k}$. The above formulation has undesirable properties from a computational perspective. The finite element spaces $V^k \cap V^{*}_k$ is difficult to construct, and the problem will not be well-posed in the presence of a non-trivial harmonic space $\mathfrak{H}^k$. In order to circumvent these issues, a well-posed (cf.~\cite{AFW06, AFW10}) \emph{mixed formulation of the abstract Hodge Laplacian} is introduced as the problem of finding $( \sigma, u, p ) \in V^{k-1} \times V^k \times \mathfrak{H}^k$, such that:
\begin{equation}\label{HL}
\begin{array}{rll} 
\langle \sigma, \tau \rangle - \langle d\tau, u \rangle& =  0, & \forall \tau \in V^{k-1},  \\ 
\langle d \sigma, v \rangle + \langle du,  dv \rangle+\langle p, v  \rangle& =  \langle f, v\rangle, & \forall v \in V^k ,\\ 
\langle u, q \rangle  &= 0, & \forall q \in \mathfrak{H}^k.
\end{array}
\end{equation}

\subsubsection*{Sub-complexes and approximate solutions to the Hodge Laplacian }

In ~\cite{AFW06,AFW10} a theory of approximate solutions to the Hodge Laplacian problem is developed by using finite dimensional approximation of Hilbert complexes. Let $(W,d)$ be a Hilbert complex with domain complex $(V,d)$. An approximating subcomplex is a set of finite dimensional Hilbert spaces, $V_h^k \subset V^k$ with the property that $dV^{k}_h \subset V_h^{k+1}$. We identify $W_h^k$ with $V_h^k$ but endowed with the norm $\ab{\cdot}{\cdot}$. Following \cite{AFW10}, we use $\mathfrak{Z}_h, \mathfrak{B}_h, \mathfrak{H}_h, \mathfrak{B}_h^*$ with obvious meaning. Since $V_h$ is a Hilbert complex, $V_h$ has a corresponding Hodge decomposition,
\begin{equation*}
V_h^k = \mathfrak{B}_h^k \oplus \mathfrak{H}_h^k \oplus \mathfrak{Z}_h^{k \perp}.
\end{equation*}
By fundamental theorem of linear algebra, we have $\mathfrak{Z}_h^\perp=\mathfrak{B}_h^*$. By this construction, $(V_h,d)$ is an abstract Hilbert complex with a well posed Hodge Laplacian problem. Find $( \sigma_h, u_h, p_h ) \in V^{k-1}_h \times V^k_h \times \mathfrak{H}^k_h$, such that
\begin{equation}\label{DHL}
\begin{array}{rll} 
\langle \sigma_h, \tau \rangle - \langle d\tau, u_h \rangle& =  0, & \forall \tau \in V^{k-1}_h,  \\ 
\langle d \sigma_h, v \rangle + \langle du_h,  dv \rangle+\langle p_h, v  \rangle& =  \langle f, v\rangle, & \forall v \in V^k _h,\\ 
\langle u_h, q \rangle  &= 0, & \forall q \in \mathfrak{H}^k_h .
\end{array}
\end{equation} 
An assumption made in~\cite{AFW10} in developing this theory is the existence of a bounded cochain projection $\pi_h:V \rightarrow V_h$, which commutes with the differential operator.

In~\cite{AFW10}, an a priori convergence result is developed for the solutions on the approximating complexes.  The result relies on the approximating complex getting sufficiently close to the original complex in the sense that inf$_{v \in V_h^k} \|u-v\|_V$ can be assumed sufficiently small for relevant $u \in V^k$.  Adaptive methods, on the other hand, gain computational efficiency by limiting the degrees of freedom used in areas of the domain where it does not significantly impact the quality of the numerical solution.

\subsection{The de Rham complex and approximation properties }
The de Rham complex is a cochain complex where the abstract results from the previous section can be applied in developing practical computational methods.  This section reviews concepts and definitions related to the de Rham complex that will be needed in our development of an adaptive finite element method.  This introduction will be brief and and mostly follows the notation from the more in-depth discussion in \cite{AFW10}.  

For the remainder of the paper we assume a 
bounded Lipschitz polyhedral domain, $\Omega \in \mathbb{R}^n, n\ge 2$.
Let $\Lambda^k(\Omega)$ be the space of smooth $k$-forms on $\Omega$, 
and $L^2\Lambda^k(\Omega)$ be the completion of $\Lambda^k(\Omega)$ with respect to the $L^2$ inner-product.
For $k=n$, the space of harmonic forms in $L^2 \Lambda^n(\Omega)$ has no nonzero element, i.e.~$\mathfrak{H}^{n}=\{0\}$, which simplifies the 
analysis in our case of interest $k = n$. However, $\sigma-\sigma_h$ is contained in $H\Lambda^{k-1}(\Omega)$, which generally contains a nontrivial harmonic component.
Note that the convergence and optimality results in~\cite{CHX} hold only for simply connected polygons in $\mathbb{R}^2$, therefore $\mathfrak{H}^{n-1}=\{0\}$ is also true in the case $k=n=2$.

\subsubsection*{ The de Rham complex }
  
Let $d$ be the exterior derivative acting as an operator from $L^2\Lambda^k(\Omega)$ to $L^2\Lambda^{k+1}(\Omega)$. We still use $\langle\cdot,\cdot\rangle$ and $\langle\cdot,\cdot\rangle_{V}$ to denote the $L^{2}$- and $V$-inner products respectively on the de Rham complex. This forms a Hilbert complex $(L^2\Lambda(\Omega), d)$, with domain complex $(H\Lambda(\Omega),d)$, where $H\Lambda^k(\Omega)$ is the set of elements in $L^2\Lambda^k(\Omega)$ with exterior derivatives in $L^2\Lambda^{k+1}(\Omega)$. The domain complex can be described with the following diagram
\begin{equation}
H\Lambda^0( \Omega) \xrightarrow{d}H\Lambda^1( \Omega) \xrightarrow{d} \cdots  \rightarrow    H\Lambda^{n-1}( \Omega ) \xrightarrow{d} L^2\Lambda^n(\Omega).
\end{equation}
It can be shown that the compactness property is satisfied, and therefore the prior results shown on abstract Hilbert complexes can be applied. 
  
The importance of the adjoint operator is clear by the first equation of the mixed Hodge Laplacian problem.  Defining the coderivative operator, $\delta: \Lambda^k(\Omega) \rightarrow \Lambda^{k-1}(\Omega)$, and two particular spaces, will be helpful in understanding the adjoint operator on the de Rham complex.
\begin{align*}
\star \delta \omega &= (-1)^k d \star \omega,\label{delta} \\
\mathring{H}\Lambda^k(\Omega) &= \{ \omega \in H\Lambda^k(\Omega) : \text{  tr }\omega = 0 \text{ on }\partial\Omega\}, \\
\mathring{H}^{*}\Lambda^k(\Omega) &:= \star \mathring{H}\Lambda^{n-k}.
\end{align*}
Combining $\delta$ with Stokes' theorem gives a useful version of integration by parts
\begin{equation}
\langle d \omega, \mu \rangle = \langle \omega, \delta \mu \rangle + \int_{\partial\Omega} \text{tr } \omega \wedge \text{tr} \star \mu, \hspace{.2cm}
 \omega \in \Lambda^{k-1}(\Omega), \text{ }\mu \in \Lambda^k(\Omega).
\end{equation}
The following result uses the above concepts and is helpful in understanding the mixed Hodge Laplace problem on the de Rham complex.
\begin{theorem}\label{afw41}(Theorem 4.1 from~\cite{AFW10})
Let d be the exterior derivative viewed as an unbounded operator $L^2\Lambda^{k-1}(\Omega) \rightarrow L^2\Lambda^k(\Omega)$ with domain $H\Lambda^k(\Omega)$.  The the adjoint d$^*$, as an unbounded operator $L^2\Lambda^k(\Omega) \rightarrow L^2\Lambda^{k-1}(\Omega)$, has $\mathring{H}^{*}\Lambda^k(\Omega) $ as its domain and coincides with the operator $\delta$.
\end{theorem}

Applying the results from the previous section and Theorem \ref{afw41}, we obtain the mixed Hodge Laplacian problem in the de Rham complex: find $(\sigma,u,p) \in H\Lambda^{k-1}(\Omega) \times H\Lambda^k(\Omega) \times \mathfrak{H}^k$ such that
\begin{equation}\label{GHDR}
\begin{aligned} 
u \perp\mathfrak{H}^k,\quad\sigma&=\delta u,\quad d\sigma + \delta d u + p= f \quad&&\text{ in } \Omega, \\
\tr\star u&= 0,\quad\tr\star du = 0\quad&&\text{ on } \partial\Omega.
\end{aligned}
\end{equation}
Using proxy fields and symmetric properties of the problem, a generic method for solving \eqref{GHDR} in the case ${k=n}$ equivalently solves Poisson's equation with homogeneous Dirichlet boundary condition. In this case $du$ = 0 and $\mathfrak{H}^n = \{0\}$. Hence the mixed Hodge Laplacian simplifies to: find $(\sigma,u) \in H\Lambda^{n-1}(\Omega) \times H\Lambda^n(\Omega)$ such that
\begin{equation}\label{HLDRC}
\begin{array}{cl} 
\sigma=\delta u,\quad d\sigma  = f &\text{ in } \Omega, \\
\text{tr} \star u = 0, &\text{ on } \partial\Omega. 
\end{array}
\end{equation}

We use $(V(\Th),d)$ [corresponds to $(V_h,d)$] to denote a finite dimensional subcomplex of $(H\Lambda,d)$ on the mesh $\Th$. Then the discrete problem \eqref{DHL} is to find  $(\sigma_h,u_h) \in V^{n-1}(\Th)\times V^n(\Th)$, such that
\begin{equation}\label{HLDDRC}
\begin{aligned}
\ab{\sigma_h}{\tau}-\ab{d\tau}{u_h}&=0,\quad&&\tau\in V^{n-1}(\Th),\\
\ab{d\sigma_h}{v}&=\ab{f}{v},&&v\in V^n(\Th).
\end{aligned}
\end{equation}
Let $\delta_h$ be the adjoint of $d: V^{n-1}(\Th)\rightarrow V^n(\Th)$, and $f_{\Th}$ be the $L^2$-projection of $f$ onto $V^{n}(\Th)$. \eqref{HLDDRC} is equivalent to $\sigma_h=\delta_h v_h,$ $ d\sigma_h=f_{\Th}$. Note that $\sigma\in\mathfrak{Z}_h^\perp$ and $\sigma\in\mathfrak{Z}^\perp$.

\subsubsection*{ Finite element differential forms }
Given a shape regular, conforming simplicial triangulation $\Th$ of $\Omega$, we set $h_T :=|T|^{\frac{1}{n}}$ for an element $T \in \mathcal{T}_h$, where $|T|$ is the volume of $T$. The finite element space $V^k(\Th)\subset H\Lambda^k(\Omega)$ is a space of $k$-forms with piecewise polynomial coefficients, 
\begin{align}
V^{n-1}(\Th)\times V^n(\Th)&=\mathcal{P}^-_{r+1}\Lambda^{n-1}\times\mathcal{P}_r\Lambda^n(\Th),\label{RT}\\
V^{n-1}(\Th)\times V^n(\Th)&=\mathcal{P}_{r+1}\Lambda^{n-1}\times\mathcal{P}_r\Lambda^n(\Th),\label{BDM}
\end{align}
$r\geq0$. In fact, $\mathcal{P}_r\Lambda^k(\Th)$ consists of $k$-forms with piecewise polynomial coefficients of degree $r$. $\mathcal{P}^-_r\Lambda^k(\Th)$ is in a special Koszul complex. Pairs \eqref{RT} and \eqref{BDM} are generalizations of RT and BDM elements respctively in FEEC. For a detailed discussion on these spaces, see \cite{AFW10}.  

\subsubsection*{ Bounded Cochain Projections}
Bounded cochain projections and their approximation properties are necessary in the analysis of both uniform and adaptive FEMs in the FEEC framework. We will use frequently the following two operators: the smoothed projection $\pi_h: L^{2}\Lambda^{k}(\Omega)\rightarrow V^{k}(\Th)$ from \cite{CW}, and the commuting quasi-interpolation $\Pi_h: L^{2}\Lambda^{k}(\Omega)\rightarrow V^{k}(\Th)$ as defined in \cite{DH} with ideas similar to \cite{sch08}.

In the remainder of the paper, $C$ will be a generic constant which is dependent only on $\Omega$ and the shape regularity of the underlying mesh . We use $\ab{\cdot}{\cdot}_{\Omega_0}$ to denote the $L^2$ inner product restricted to $\Omega_0$. $\| \cdot \|$ will denote the $L^2\Lambda^k(\Omega)$ norm, and when taken on specific elements of the domain $T$ and $\partial T$, we write $\|\cdot\|_T$ and $\|\cdot\|_{\partial T}$ respectively.  For all other norms, such as $H\Lambda^k(\Omega)$ and $H^1\Lambda^k(\Omega)$, we write $\|\cdot\|_{H\Lambda^k(\Omega)}$ and $\|\cdot\|_{H^1\Lambda^k(\Omega)}$ respectively.

The next lemma is taken directly from Lemma 6 in \cite{DH}, and will be a key tool in developing an upper bound for the error.

\begin{lemma}\label{dhi}

Assume 1 $\le k \le $ n, and $\phi \in H\Lambda^{k-1}(\Omega)$ with $\| \phi \|_{H\Lambda^{k-1}(\Omega)}\le 1.$ 
 Then there exists $\varphi \in H^1 \Lambda^{k-1}(\Omega)$ such that
 $d\varphi = d\phi, \Pi_h d \phi = d\Pi_h \phi = d\Pi_h \varphi$, and

\begin{equation*}
 \displaystyle\sum\limits_{T \in T_h}  h_T^{-2} \| \varphi - \Pi_h\varphi \|_{T} ^2 +   h_T^{-1} \| \emph{tr} ( \varphi - \Pi_h\varphi ) \|_{\partial T }^2 \le C.
\end{equation*}
\end{lemma}

\section{Error Estimator}
\label{sec:estimator}
For $T \in \Th$, let $\lr{\tau}$ denote the jump of $\tau$ over an element face. For element faces on $\partial \Omega$ we set $\lr{\tau} = \tau$. The element error indicator is defined as
\begin{equation*}
\eta_{\Th}^2(\sigma_h,T) =  h_T \|\lr{\tr\star\sigma_h}\|_{\partial T}^2 +  h_T^2\|  \delta \sigma_h\|^2_T+  h_T^2\|  f - f_{\Th}\|^2_T.
\end{equation*}

For a subset $\mathcal{M}\subseteq\Th$, define
\begin{align*}
\eta_{\Th}^2( \sigma_h,\mathcal{M})&:=\sum_{T \in \mathcal{M}}\eta_{\Th}^2(\sigma_h,T),\\
\osc_{\Th}^2(f,\mathcal{M})&:=\sum_{T\in\mathcal{M}}h_T^2\|f-f_{\Th}\|^2_T.  
\end{align*}

The Hodge decomposition is crucial to proving global reliability of $\eta_{\Th}$. By $\sigma\in\mathfrak{Z}^\perp$ and $\sigma_h\in\mathfrak{Z}_h^\perp$, the Hodge decomposition of $\sigma - \sigma_h$ can be written as 
\begin{equation}\label{Hodgesigma}
\begin{aligned}
\sigma-\sigma_h&=P_{\mathfrak{B}}(\sigma-\sigma_h)+P_{\mathfrak{H}}(\sigma-\sigma_h)+P_{\mathfrak{Z}^\perp}(\sigma-\sigma_h)\\
&=(\sigma-P_{\mathfrak{Z}^\perp}\sigma_h)-P_{\mathfrak{B}}\sigma_h-P_{\mathfrak{H}}\sigma_h.
\end{aligned}
\end{equation}
Lemmas \ref{part1}, \ref{part2} and \ref{part3} will bound each portion of this orthogonal decomposition.  

\begin{lemma}\label{part1}
\begin{equation*} \|\sigma-P_{\mathfrak{Z}^\perp}\sigma_h\|\leq C\osc_{\Th}(f,\Th). 
\end{equation*}
\end{lemma}

\begin{proof}
Since $\sigma-P_{\mathfrak{Z}^\perp}\sigma_h\in\mathfrak{Z}^{n-1\perp}=\mathfrak{B}_{n-1}^*$, 
$\sigma-P_{\mathfrak{Z}^\perp}\sigma_h=\delta v$ for some $v\in\text{Dom}(\delta)=H_0^1(\Omega)$.
Thus
\begin{equation*}
\begin{aligned}
\|\sigma-P_{\mathfrak{Z}^\perp}\sigma_h\|^2&=\ab{\sigma-P_{\mathfrak{Z}^\perp}\sigma_h}{\delta v}=\ab{d\sigma-d\sigma_h}{v}.
\end{aligned}
\end{equation*}
Then by $\sum_{T\in\Th}h_T^{-2}\|v-v_{\Th}\|_T^2\leq C\|\delta v\|^2$, we obtain
\begin{equation*}
\begin{aligned}
\|\sigma-P_{\mathfrak{Z}^\perp}\sigma_h\|^2=\ab{f-f_{\Th}}{v-v_{\Th}}
\leq C\osc_{\Th}(f,\Th)\|\delta v\|.
\end{aligned}
\end{equation*}
The proof is complete.
\end{proof}


The next lemma uses the quasi-interpolant $\Pi_h$ described in \cite{DH}, and also applies integration by parts in the same standard fashion that \cite{DH} use when bounding error measured in the natural norm, $\|u - u_h\|_{H\Lambda^k(\Omega)} + \|\sigma - \sigma_h\|_{H\Lambda^{k-1}(\Omega)} + \|p - p_h\|$.  In~\cite{DH}, $\inf$-$\sup$ condition of the bilinear-form is used to separate components of the error, whereas here we simply analyze the orthogonal decomposition of $\sigma - \sigma_h$.  

\begin{lemma}\label{part2} 
\begin{equation*}
\|P_{\mathfrak{B}} \sigma_h\|\leq C\eta_{\Th}(\sigma_h,\Th).
\end{equation*}
\end{lemma}
\begin{proof} 
\begin{equation*}
\| P_{\mathfrak{B}} \sigma_h\| = \ab{\sigma_h}{P_{\mathfrak{B}}\sigma_h/\| P_{\mathfrak{B}} \sigma_h \|}=\ab{-\sigma_h}{d\phi}, \quad\phi\in(\mathfrak{Z}^{k-2})^{\perp V}. 
\end{equation*}
Since $\phi$ can then be replaced with $\varphi$ satisfying the properties of Lemma \ref{dhi}, and noting $\sigma_h\perp\mathfrak{B}^{k-1}_h$,
\begin{equation}\label{dB}
\|P_{\mathfrak{B}}\sigma_{h}\|= \ab{- \sigma_h}{d (\varphi - \Pi_h\varphi)}.
\end{equation}
The problem is now reduced to a case handled in \cite{DH}, when they bound a portion of their $\eta_{-1}$ estimator.  We follow their ideas to complete to proof.  Applying the integration by parts formula we have
\begin{equation*}
 \|P_{\mathfrak{B}}\sigma_{h}\|=\sum_{T\in\Th}\int_{\partial T}(\tr\star \sigma_h\wedge\tr(\varphi - \Pi_h\varphi ) ) +\ab{\delta\sigma_h}{\varphi - \Pi_h\varphi}_T. 
 \end{equation*}
Noting tr($\varphi - \Pi_h\varphi$) is single-valued on the element boundaries, this can be reduced to 
\begin{align*}
\|P_{\mathfrak{B}}\sigma_{h}\|&\leq C\sum\limits_{T \in\Th}   \|\tr( \varphi - \Pi_h\varphi ) \|_{\partial T } \|\lr{\tr\star \sigma_h}\|_{\partial T } + \| \varphi - \Pi_h\varphi \|_{T} \| \delta \sigma_h\|_{T }\\
&\leq C\sum_{T \in\Th} \big(h_T^{\frac{1}{2}}\|\lr{\tr\star \sigma_h}\|_{\partial T } +  h_T\| \delta \sigma_h\|_{T }\big)\\
&\quad\times\big( h_T^{-\frac{1}{2}} \|\tr( \varphi - \Pi_h\varphi ) \|_{\partial T} + h_T^{-1}\| \varphi - \Pi_h\varphi \|_{T}\big)\\
&\leq C \eta_{\Th}(\sigma_h, \Th)\big(\sum_{T\in\Th} h_T^{-1} \|\tr( \varphi - \Pi_h\varphi ) \|^2_{\partial T } + h_T^{-2}\| \varphi - \Pi_h\varphi \|^2_{T}\big)^{1/2}.
\end{align*}
The proof is then complete by applying the bounds from Lemma \ref{dhi}, and the Poincar\'e inequality $\| \phi \|_{H\Lambda^{k-1}}\leq C\|d\phi\|=C.$ 
\end{proof}

To control the harmonic component in the Hodge decomposition, we need to estimate the gap between $\mathfrak{H}$ and $\mathfrak{H}_{h}$. To this end, we use equation (28) in \cite{AFW10}:
\begin{equation}\label{hIntAFW}  
\|(I-P_{\mathfrak{H}^k})q \|_V \le \| (I - \pi_h^k)P_{\mathfrak{H}^k}q \|_V, \hspace{.5cm} q\in \mathfrak{H}_h^k.
\end{equation}
We mention that $\|\tilde{q}\|=\|\tilde{q}\|_{V}$ for any $\tilde{q}\in\mathfrak{H}^{k}$ or $\mathfrak{H}_{h}^{k}$. 
Combining \eqref{hIntAFW} with a triangle inequality gives
\begin{equation}\label{AFW3}
\|q\|  \le (\| (I - \pi_h^k)\|+1) \|P_{\mathfrak{H}^k}q \|\leq C\|P_{\mathfrak{H}^k }q \|, \hspace{.5cm} q\in \mathfrak{H}_h^k.
\end{equation}
Theorem \ref{chdt} will be essential in dealing with the harmonic forms in the proof of a continuous upper-bound.  The corollary will be used identically when proving a discrete upper-bound.  For use in our next two results we introduce an operator $\delta$ and one of its important properties.  Let $A, B$ be $n < \infty$ dimensional, closed subspaces of a Hilbert space $W$, and let
\begin{equation*}
\delta( A, B ) = \sup_{x\in A,\ \|x\| = 1} \|x - P_Bx \|,
\end{equation*}
then \cite{DH}, Lemma 2 which takes the original ideas from \cite{KATO}, shows
\begin{equation}\label{DHL2}
\delta(A,B) = \delta(B, A ).
\end{equation}

\begin{theorem}\label{chdt}
For $1\leq k\leq n-1$,
\begin{equation*}
 \delta( \mathfrak{H}^k,\mathfrak{H}^k_h)=\delta( \mathfrak{H}^k_h, \mathfrak{H}^k ) \leq C<1.
\end{equation*}
\end{theorem}
\begin{proof}
 $\text{dim}\mathfrak{H}^k_H=\text{dim}\mathfrak{H}^k=\beta_{k}$, the $k$th Betti number of the domain $\Omega$. Then we can apply \eqref{DHL2} to prove the equality. By \eqref{AFW3} and the orthogonality of the $L^{2}$-projection, we have
\begin{align*}
\delta( \mathfrak{H}^k_h, \mathfrak{H}^k ) &= \sup_{q\in \mathfrak{H}^k_h,\ \|q\|=1 } \| q - P_{\mathfrak{H}} q \|\\
&= \sup_{q\in \mathfrak{H}^k_h,\ \|q\|=1 } \sqrt{1-\|P_{\mathfrak{H}} q \|^{2}}\\
&\leq\sqrt{ 1- \frac{1}{C^2} } < 1.
\end{align*}
The proof is complete.
\end{proof}

\begin{corollary}\label{dhdc}
Let $\Th$ be a conforming refinement of $\mathcal{T}_{H}$. Then 
\begin{equation*} 
 \delta( \mathfrak{H}_h^k, \mathfrak{H}_H^k ) = \delta( \mathfrak{H}_H^k, \mathfrak{H}_h^k ) \le C< 1 .
\end{equation*}
\begin{proof}
The proof follows the same logic as Theorem \ref{chdt}. The only difference is replacing
\eqref{hIntAFW}  
by
\begin{equation*}
\|(I-P_{\mathfrak{H}_{h}})q\|_V\leq\|(I-\pi_H^k)P_{\mathfrak{H}_{h}}q \|_V,\quad q\in\mathfrak{H}_{H}^k,
\end{equation*}
which can be derived by following the proof of \eqref{hIntAFW}. 
\end{proof}
\end{corollary}

\begin{lemma}\label{part3}
\begin{equation*}
\|P_{\mathfrak{H}} \sigma_h\|\leq C_{\mathfrak{H}}\|\sigma-\sigma_h\|, \quad C_{\mathfrak{H}}< 1.
\end{equation*}
\end{lemma}
\begin{proof}  Since $\sigma\perp \mathfrak{Z}^{k-1}$ and $\sigma_h \perp \mathfrak{Z}_h^{k-1}$, we have
\begin{align*}
\|P_{\mathfrak{H}}\sigma_h\|&= \sup _{q \in \mathfrak{H}, \|q\| = 1}\langle\sigma_h, q - P_{\mathfrak{H}_h}q\rangle \\
&= \sup _{q \in \mathfrak{H}, \|q\| = 1}\langle\sigma_h -\sigma, q - P_{\mathfrak{H}_h}q\rangle \\
&\le\delta( \mathfrak{H}^{n-1},\mathfrak{H}^{n-1}_h)  \| \sigma_h- \sigma\| .
\end{align*}
Then Lemma \ref{part3} follows from Theorem $\ref{chdt}$.
\end{proof}

Now we are in a position to prove the continuous upper bound.
\begin{theorem}\label{upper}\emph{(continuous upper bound)} There exists a constant $C_1$, depending only on the shape regularity of $\Th$, such that
\begin{equation*} 
\| \sigma - \sigma_h \|^2\leq C_1\eta_{\Th}^2(\sigma_h,\Th). 
\end{equation*}
\end{theorem}
\begin{proof}
Starting from \eqref{Hodgesigma}, by Lemmas \ref{part1}, \ref{part2} and \ref{part3}, we have
\begin{align*}
 \| \sigma - \sigma_h\|&\leq\|\sigma -P_{\mathfrak{Z}^\perp}\sigma_h\|+\|P_{\mathfrak{H}}\sigma_{h}\|+\|P_{\mathfrak{B}}\sigma_{h}\|\\
&\leq\frac{1}{1-C_{\mathfrak{H}}}(\|\sigma-P_{\mathfrak{Z}^\perp}\sigma_h\|+\|P_{\mathfrak{B}}\sigma_h\|)\\
&\leq C_1\eta_{\Th}(\sigma_h,\mathcal{T}_h).
\end{align*}
The proof is complete.
\end{proof}

The efficiency can be proved by the standard bubble function technique in \cite{DH}. 
\begin{theorem}\label{lower}\emph{(lower bound)}
There exists a constant $C_2$ depending only on the shape regularity of $\Th$, such that
\begin{equation*} C_2\eta_{\Th}^2( \sigma_h, \mathcal{T}_h ) \le \| \sigma - \sigma_h \|^2+ \osc_{\Th}^2(f,\mathcal{T}_h ) . \end{equation*}
\end{theorem}

\section{Quasi-orthogonality}\label{sec:quasiortho}
The main difficulty for proving convergence of AMFEM is the failure of orthogonality. In \cite{CHX}, a quasi-orthogonality property is proven using a technical discrete stability result. In this section, we use a novel technique to prove a sharper quasi-orthogonality result on $\ab{\sigma-\sigma_h}{\sigma_h-\sigma_H}$. In the next lemma, we prove a discrete approximation result.
\begin{lemma}\label{disapprox}
Let $\Th$ be a conforming refinement of $\mathcal{T}_H$ and $P_H$ be the $L^2$ projection onto $\mathcal{P}_0\Lambda^n(\mathcal{T}_H)$. Then for any $T\in\mathcal{T}_H$ and $v_h\in V^n(\Th)$,    
\begin{equation*} 
\|v_h -P_H v_h\|_{T}\leq C h_T\|\delta_h v_h\|_{T}. 
\end{equation*}
\end{lemma}
\begin{proof}
We prove it by homogeneity argument. Suppose $\delta_h v_h=0$ on $T$. Let $\hat{\mathcal{T}}_h:=\{t\in\mathcal{T}_h: t\subset T\},$ and  $V^{n-1}(\hat{\mathcal{T}}_h):=\{\tau\in V^{n-1}(\mathcal{T}_h): \supp\tau\subseteq T\}$. Let $\mathcal{E}_h$ denote the set of faces of $\Th$,  $\hat{\mathcal{E}}_h$ the set of faces of $\mathcal{T}_h$ in the interior of $T$. Then for any $\tau_h\in V^{n-1}(\hat{\mathcal{T}}_h)$, 
\begin{equation*}
\ab{d\tau_{h}}{v_h}=\ab{\tau_{h}}{\delta_{h}v_h}=\ab{\tau_{h}}{\delta_{h}v_h}_T=0.
\end{equation*}
By element-wise integration by parts and the property of Hodge star, we have
\begin{equation}\label{test}
\begin{aligned}
0&=\sum_{t\in\hat{\mathcal{T}}_{h}}\int_{\partial t}\tr\tau_h\wedge\tr\star v_h+\ab{\tau_h}{\delta v_h}_{t}\\
&=\sum_{e\in\hat{\mathcal{E}}_h}\int_e\tr\tau_h\wedge\lr{\tr\star v_h}+\sum_{t\in\hat{\mathcal{T}}_{h}}\int_{t}\tau_h\wedge\star\delta v_h.
\end{aligned}
\end{equation}
In the last equality, we use $\ab{\tau_h}{\delta v_h}_t=\ab{\star\tau_h}{\star\delta v_h}_t=\int_t\tau_h\wedge\star\delta v_h$. First assume that
$$V^{n-1}(\Th)\times V^n(\Th)=\mathcal{P}_{r+1}^{-}\Lambda^{n-1}(\Th)\times\mathcal{P}_{r}\Lambda^{n}(\Th),\quad r\geq0$$ 
is the generalized RT pair. Then $v_h\in \mathcal{P}_r\Lambda^n(\Th).$ The degrees of freedom for $V^{n-1}(\Th)$ (cf.\cite{AFW06}) are given by
\begin{align}
&\int_{e}\tr\tau_h\wedge\mu,\quad e\in\mathcal{E}_h,\quad\mu\in\mathcal{P}_{r}\Lambda^{0}(e),\label{degree1}\\
&\int_{t}\tau_h\wedge\mu,\quad t\in\Th,\quad\mu\in\mathcal{P}_{r-1}\Lambda^{1}(t).\label{degree2}
\end{align}
Corresponding to the above degrees of freedom, let 
$$\big(\bigcup_{e\in\Eh}\{\tau_{e}^1,\cdots,\tau_{e}^{m_e}\}\big)\bigcup\big(\bigcup_{t\in\Th}\{\tau_{t}^1,\cdots,\tau_{t}^{m_t}\}\big)$$ 
be a dual basis, where $m_e, m_t\geq0$ are integers depending on $e$ and $t$. In particular, 
$$\int_{e^\prime}\tr\tau_e^i\wedge\mu=0,\quad\text{ for all }\mu\in\mathcal{P}_{r}\Lambda^{0}(e^\prime)\text{ and }e^\prime\neq e,$$ 
and degrees of freedom of $\tau_f^i$ in \eqref{degree2} vanish for all $t\in\Th$; 
$$\int_{t^\prime}\tau_t^j\wedge\mu=0,\quad\text{ for all }\mu\in\mathcal{P}_{r-1}\Lambda^{1}(t^\prime)\text{ and }t^\prime\neq t,$$
and degrees of freedom of $\tau_t^j$ in \eqref{degree1} vanish for all $e\in\Eh$. 

For $e\in\hat{\mathcal{E}}_h$, let $\tau_{h}$ in \eqref{test} go through $\tau_e^1,\cdots,\tau_e^{m_e}$. Since $\lr{\tr\star v_h}|_{e^\prime}\in\mathcal{P}_{r}\Lambda^{0}(e^\prime)$ and $\star\delta v_h|_{t^\prime}\in\mathcal{P}_{r-1}\Lambda^{1}(t^\prime),$ we have
$$\int_e\tr\tau_e^i\wedge\lr{\tr\star v_h}=0,\quad i=1,\cdots,m_e.$$
Then by $\lr{\tr\star v_h}|_e\in\mathcal{P}_{r}\Lambda^{0}(e)$ and unisolvence, $\lr{\tr\star v_h}=0$ on any $e\in\hat{\mathcal{E}}_h$ and thus $v_h$ is continuous on $T$. 

On the other hand, for $t\in\hat{\mathcal{T}}_h$, let $\tau_{h}$ in \eqref{test} go through $\tau_t^1,\cdots,\tau_t^{m_t}$. For the same reason, we have
$$\int_t \tau_t^j\wedge\star\delta v_h=0,\quad j=1,\cdots,m_t.$$
Then by $\star\delta v_h|_t\in\mathcal{P}_{r-1}\Lambda^1(t)$ and unisolvence, $\delta v_h=\star\delta v_h=0$ on any $t\in\hat{\mathcal{T}}_h$ and thus $v_h$ is a piecewise constant in $T$. 

Hence $v_h$ is a constant on $T$. In summary, $\delta_h v_h=0$ on $T$ implies $v_h-P_H v_h=0$ on $T$. \eqref{disapprox} then follows from the Bramble--Hilbert lemma, affine equivalence between $T$ and a reference triangle, and the shape regularity of $\Th$. The same argument applies to the generalized BDM pair 
$$V^{n-1}(\Th)\times V^n(\Th)=\mathcal{P}_{r+1}\Lambda^{n-1}(\Th)\times\mathcal{P}_{r}\Lambda^{n}(\Th),\quad r\geq0.$$ 
The proof is complete.
\end{proof} 

The quasi-orthogonality result is a direct corollary of Lemma \ref{disapprox}.
\begin{theorem}\label{quasiortho}  
Let $\mathcal{T}_h$ be a refinement of $\mathcal{T}_H$ and $\mathcal{R}_H$ be the set of refined elements in $\mathcal{T}_H$. Then
for any $\varepsilon > 0$,
\begin{equation*}
(1-\varepsilon)\|\sigma - \sigma_h\|^2 \le \| \sigma - \sigma_H \|^2 - \| \sigma_h - \sigma_H \|^2 + \frac{C_0}{\varepsilon}\osc_{\mathcal{T}_H}^2(f,\mathcal{R}_H).
\end{equation*}
\end{theorem}
\begin{proof}
By $\sigma\perp\mathfrak{Z}_{h}$ and $\sigma_h\perp\mathfrak{Z}_h$, we have
\begin{equation}\label{inner}
\begin{aligned}
|\ab{\sigma - \sigma_h}{\sigma_h - \sigma_H}|&=|\ab{\sigma - \sigma_h}{P_{\mathfrak{Z}_h^{\perp}}(\sigma_h - \sigma_H)}|\\
&\le \| \sigma - \sigma_h \| \|P_{\mathfrak{Z}_h^{\perp}}(\sigma_h - \sigma_H)\|.
\end{aligned}
\end{equation}
Since $P_{\mathfrak{Z}_h^{\perp}}(\sigma_h - \sigma_H)\in\mathfrak{Z}_{h}^{\perp}=\mathfrak{B}_{h}^*$, there exists some $v_h\in V^n(\Th)$, such that $P_{\mathfrak{Z}_h^{\perp}}(\sigma_h - \sigma_H)=\delta_h v_h$. Then by $\sigma_h-\sigma_H=d\delta_h v_h$ and $\mathcal{P}_0\Lambda^n(\Th)\subseteq V^n(\mathcal{T}_h)$,
\begin{equation}\label{interquasiortho}
\begin{aligned}
\|P_{\mathfrak{Z}_h^{\perp}}(\sigma_h - \sigma_H)\|^2&=\ab{\delta_h v_h}{\delta_h v_h}\\
&=\ab{d(\sigma_h - \sigma_H)}{v_h}\\
&=\ab{f_{\Th}-f_{\mathcal{T}_H}}{v_h-P_H v_h},\\
&=\ab{f-f_{\mathcal{T}_H}}{v_h-P_H v_h}.
\end{aligned}
\end{equation}
For $T\in\mathcal{T}_H\backslash\mathcal{R}_H$, $v_h=P_H v_h$ on $T$. Hence by \eqref{interquasiortho}, Cauchy--Schwarz inequality, and Lemma \ref{disapprox}, we have
\begin{equation}\label{estimatezperp}
\begin{aligned}
\|P_{\mathfrak{Z}_h^{\perp}}(\sigma_h - \sigma_H)\|^2&=\sum_{T\in\mathcal{R}_H}\int_{T}(f-f_{\mathcal{T}_H})(v_h-P_H v_h)\\
&\leq\osc_{\mathcal{T}_H}(f,\mathcal{R}_H)\left(\sum_{T\in\mathcal{R}_H}h_T^{-2}\|v_h-P_H v_h\|_T^2\right)^{\frac{1}{2}}\\
&\leq C_0^\frac{1}{2}\osc_{\mathcal{T}_H}(f,\mathcal{R}_H)\|\delta_h v_h\|.
\end{aligned}
\end{equation}
It then follows from \eqref{inner} and \eqref{estimatezperp} that
\begin{equation*}
    \begin{aligned}\|\sigma-\sigma_{h}\|^{2}&=\|\sigma-\sigma_{H}\|^{2}-\|\sigma_h-\sigma_{H}\|^{2}-2\langle\sigma-\sigma_{h},\sigma_{h}-\sigma_{H}\rangle\\
    &\leq\|\sigma-\sigma_{H}\|^{2}-\|\sigma_h-\sigma_{H}\|^{2}+\varepsilon\|\sigma-\sigma_{h}\|^2+\varepsilon^{-1}C_0\osc^2_{\mathcal{T}_H}(f,\mathcal{R}_H).
    \end{aligned}
\end{equation*}
The proof is complete.
\end{proof}
Comparing to the quasi-orthogonality 
\begin{equation}\label{CHXquasi}
(1-\varepsilon)\|\sigma - \sigma_h\|^2 \le \| \sigma - \sigma_H \|^2 - \| \sigma_h - \sigma_H \|^2 + \frac{C}{\varepsilon}\osc_{\mathcal{T}_H}^2(f,\mathcal{T}_H)
\end{equation}
proved in \cite{CHX}, Theorem \ref{quasiortho} is sharper because $\osc_{\mathcal{T}_H}(f,\mathcal{R}_H)\leq\osc_{\mathcal{T}_H}(f,\mathcal{T}_H).$ This improvement is crucial to the convergence analysis. Replacing $\osc_{\mathcal{T}_H}(f,\mathcal{T}_H)$ by $\osc_{\mathcal{T}_H}(f,\mathcal{R}_H)$ is motivated by the quasi-orthogonality result in \cite{BM} for the lowest order RT mixed method on simply connected polygon in $\mathbb{R}^2$. However, our technique is applicable to general domains in $\mathbb{R}^n$ and quite different from \cite{BM} as well as \cite{CHX}.

\section{Convergence and Optimality}
\label{sec:conv}
Given an initial triangulation, $\mathcal{T}_0$, the adaptive procedure will generate a nested sequence of triangulations $\mathcal{T}_\ell$ and discrete solutions $\sigma_\ell$ and $u_\ell$, by looping through the following steps:
$$
 \textsf{Solve} \longrightarrow \textsf{Estimate} \longrightarrow \textsf{Mark} \longrightarrow \textsf{Refine}$$
Our adaptive mixed finite element method is as follows.

$[\mathcal{T}_N, \sigma_N]=\textsf{AMFEM}(f,\mathcal{T}_0,\theta, \tol)$

Given an initial mesh $\mathcal{T}_0$, a marking parameter $0<\theta<1$, an error tolerance $\tol>0$. Set $\ell=0$, $\eta_{\ell}=\tol>0$. 

\textbf{WHILE} $\eta_\ell\geq\tol$, \textbf{DO}
\begin{center}
\begin{itemize}
\item[\qquad\textbf{1}.]Solve the discrete problem \eqref{HLDDRC} on $\mathcal{T}_\ell$ to obtain the solution $\sigma_\ell$.
\item[\qquad\textbf{2}.]For each $T\in\mathcal{T}_\ell$, compute $\eta_{\mathcal{T}_\ell}(\sigma_l,T)$ and $\eta_\ell=\eta_{\mathcal{T}_\ell}(\sigma_\ell,\mathcal{T}_\ell)$.
\item[\qquad\textbf{3}.]Select a subset $\mathcal{M}_\ell$ of $\mathcal{T}_\ell$ such that $\eta_{\mathcal{T}_l}(\sigma_\ell,\mathcal{M}_\ell)\geq\theta\eta_{\mathcal{T}_\ell}(\sigma_\ell,\mathcal{T}_\ell).$
\item[\qquad\textbf{4}.]Refine $\mathcal{T}_\ell$ and necessary neighboring simplices by newest vertex bisection to get a conforming $\mathcal{T}_{\ell+1}$. Set $\ell\leftarrow\ell+1$ and go to Step 1. 
\end{itemize}
\end{center}

\textbf{END DO}

$\mathcal{T}_N=\mathcal{T}_\ell,\quad$ $\sigma_N=\sigma_\ell$.

The simple newest vertex bisection can maintain the shape regularity of $\{\mathcal{T}_\ell\}$, i.e., $\mathcal{T}_\ell$ is shape regular and the shape regularity depends only on $\mathcal{T}_0$. Bounding the number of simplexes generated in mesh refinements is important in the proof of quasi optimality. By choosing a suitable initial labeling, Stevenson \cite{RS} showed that newest vertex bisection guarantees
\begin{align}\label{NVB}
\# \mathcal{T}_\ell\leq\#\mathcal{T}_0 + C\sum_{i=0}^{\ell-1}\#\mathcal{M}_i.
\end{align}

\subsection{ Convergence of \textsf{AMFEM}}
This subsection is devoted to convergence analysis of \textsf{AMFEM}. The results in this section follow ideas already in the literature \cite{CKNS,CHX,BM}, with Theorem \ref{termination} building on these ideas by proving reduction in a total error using relationships between data oscillation and reduction of a second type of total error. The following notation will be used in the proofs and discussion of this section:
\begin{equation*}
\begin{aligned}
&e_\ell=\|\sigma- \sigma_\ell\|^2,
\quad
E_\ell=\|\sigma_{\ell+1}- \sigma_\ell\|^2, 
\quad 
\eta_\ell =\eta_{\mathcal{T}_\ell}^2(\sigma_\ell,\mathcal{T}_\ell),\\
&o_\ell= \osc^2(f,\mathcal{T}_\ell),\quad\hat{o}_\ell= \osc^2(f,\mathcal{R}_\ell).
\end{aligned}
\end{equation*}
where $\mathcal{R}_\ell$ is the set of refined elements in $\mathcal{T}_\ell$.

\begin{lemma}
\begin{equation}\label{cts}
\eta_{\ell+1}\leq\beta\eta_\ell+C_3 E_\ell,
\end{equation}
where $0<\beta<1$ and $C_3>0$ depend only on $\theta$ and $\mathcal{T}_0$.
\end{lemma}
\begin{proof}
The proof is similar to Corollary 3.4 in \cite{CKNS}. Since $\eta_\ell$ involves data oscillation, we sketch the proof here for clarity. Let $\eta_\ell=\hat{\eta}_{\ell}+o_\ell$, where $\hat{\eta}_{\ell}=\hat{\eta}^2_{\mathcal{T}_\ell}(\sigma_\ell,\mathcal{T}_\ell)$ is the standard estimator without data oscillation. Given $T\in\mathcal{T}_{\ell+1}$, using a Young's inequality with parameter $\delta_*>0$, we have
\begin{equation*}
\hat{\eta}^2_{\mathcal{T}_{\ell+1}}(\sigma_{\ell+1},T)\leq(1+\delta_*)\hat{\eta}^2_{\mathcal{T}_{\ell+1}}(\sigma_{\ell},T)+(1+\delta_*^{-1})C_{\mathcal{T}_0}\|\sigma_{\ell+1}-\sigma_{\ell}\|^2_T,
\end{equation*}
Summing over $T\in\mathcal{T}_{\ell+1}$ gives
\begin{equation*}
\hat{\eta}^2_{\mathcal{T}_{\ell+1}}(\sigma_{\ell+1},\mathcal{T}_{\ell+1})\leq(1+\delta_*)\hat{\eta}^2_{\mathcal{T}_{\ell+1}}(\sigma_{\ell},\mathcal{T}_{\ell+1})+(1+\delta_*^{-1})C_{\mathcal{T}_0}\|\sigma_{\ell+1}-\sigma_{\ell}\|^2,
\end{equation*}
and thus 
\begin{equation}\label{ctsinter1}
{\eta}_{\ell+1}\leq(1+\delta_*){\eta}^2_{\mathcal{T}_{\ell+1}}(\sigma_{\ell},\mathcal{T}_{\ell+1})+(1+\delta_*^{-1})C_{\mathcal{T}_0}E_\ell.
\end{equation}
For $T\in\mathcal{T}_\ell$, we use $\hat{\mathcal{T}}(T)=\{t\in\mathcal{T}_{\ell+1}: t\subset T\}.$ If $T\in\mathcal{M}_\ell$ is marked, 
\begin{equation}\label{ctsinter2}
\sum_{t\in\hat{\mathcal{T}}(T)}\hat{\eta}_{\mathcal{T}_{\ell+1}}^2(\sigma_\ell,t)\leq2^{-\frac{1}{n}}\hat{\eta}_{\mathcal{T}_\ell}^2(\sigma_\ell,T),
\end{equation}
see Corollary 3.4 in \cite{CKNS}; and
\begin{equation}\label{ctsinter3}
\sum_{t\in\hat{\mathcal{T}}(T)}\osc_{\mathcal{T}_{\ell+1}}^2(f,t)\leq2^{-\frac{2}{n}}\osc_{\mathcal{T}_\ell}^2(f,T),
\end{equation}
see Lemma \ref{oscreductionthm}. If $T\in\mathcal{T}_{\ell}\backslash\mathcal{M}_\ell$, we use $$\sum_{t\in\hat{\mathcal{T}}(T)}\hat{\eta}_{\mathcal{T}_{\ell+1}}(\sigma_\ell,t)\leq\hat{\eta}_{\mathcal{T}_\ell}(\sigma_\ell,T),\quad\sum_{t\in\hat{\mathcal{T}}(T)}\osc_{\mathcal{T}_{\ell+1}}(f,t)\leq\osc_{\mathcal{T}_\ell}(f,T).$$ Combining the above inequality with \eqref{ctsinter2} and \eqref{ctsinter3}, we obtain
\begin{equation}\label{ctsinter4}
\begin{aligned}
{\eta}^2_{\mathcal{T}_{\ell+1}}(\sigma_{\ell},\mathcal{T}_{\ell+1})&\leq{\eta}^2_{\mathcal{T}_{\ell}}(\sigma_{\ell},\mathcal{T}_{\ell}\backslash\mathcal{M}_\ell)+2^{-\frac{1}{n}}{\eta}^2_{\mathcal{T}_{\ell}}(\sigma_{\ell},\mathcal{M}_{\ell})\\
&={\eta}^2_{\mathcal{T}_{\ell}}(\sigma_{\ell},\mathcal{T}_{\ell})-\lambda{\eta}^2_{\mathcal{T}_{\ell}}(\sigma_{\ell},\mathcal{M}_{\ell}),
\end{aligned}
\end{equation}
where $\lambda=1-2^{-\frac{1}{n}}<1$. It then follows from \eqref{ctsinter1} and \eqref{ctsinter4} that
\begin{equation}\label{ctseta}
\begin{aligned}
\eta_{\ell+1}\leq(1+\delta_{*})\big(\eta_{\ell}-\lambda\eta^2_{\mathcal{T}_\ell}(\sigma_\ell,\mathcal{M}_{\ell})\big)+(1+\delta_{*}^{-1})C_{\mathcal{T}_0}E_{\ell}.
\end{aligned}
\end{equation}
Combining \eqref{ctseta} and the marking property $\eta^2_{\mathcal{T}_\ell}(\sigma_\ell,\mathcal{M}_{\ell})\geq\theta^2\eta_\ell^2$, we obtain \eqref{cts} with $\beta=(1+\delta_*)(1-\lambda\theta^2)$. $\beta<1$ provided $\delta_*<\frac{\lambda\theta^2}{1-\lambda\theta^2}$.
\end{proof}

Now we are in a position to prove the error reduction.
\begin{theorem}\label{errorreduction}
When
\begin{equation*}
0<\varepsilon<\frac{1-\beta}{C_1 C_3},
\end{equation*}
there exists $\alpha \in$ \emph{(0,1)} and $C_4, \rho>0$ depending only on $\varepsilon$, $\theta$ and $\mathcal{T}_0$, such that
\begin{equation}\label{alpha}
(1- \varepsilon)e_{\ell+1} + \rho\eta_{\ell+1} \leq\alpha [(1-\varepsilon)e_\ell+\rho\eta_\ell] + C_4\hat{o}_\ell.
\end{equation}
\end{theorem}
\begin{proof}
Recall the quasi-orthogonality Theorem \ref{quasiortho} and global reliability Theorem \ref{upper}, 
\begin{align}
e_\ell&\leq C_1\eta_\ell,\label{up}\\
(1-\varepsilon)e_{\ell+1} &\leq e_\ell - E_\ell +C_0\varepsilon^{-1}\hat{o}_\ell, \text{ for any } \varepsilon> 0.\label{qo}
\end{align}
Let $\rho=1/C_3$ and $\alpha\in(0,1)$ to be determined. It follows from \eqref{qo}, \eqref{cts}, and \eqref{up} that
\begin{equation*}
\begin{aligned}
\left(1-\varepsilon\right)e_{\ell+1}+\rho\eta_{\ell+1}&\leq e_\ell+\rho\beta\eta_\ell+C_0\varepsilon^{-1}\hat{o}_\ell,\\
&\leq\alpha(1-\varepsilon)e_\ell+\left\{[1-\alpha(1-\varepsilon)]C_1+\rho\beta\right\}\eta_\ell+C_0\varepsilon^{-1}\hat{o}_\ell.
\end{aligned}
\end{equation*}
Let $\alpha$ solve
$\alpha\rho=[1-\alpha(1-\varepsilon)]C_1+\rho\beta.$
Then we obtain $$\alpha=\frac{C_1+\rho\beta}{(1-\varepsilon)C_1+\rho}<1$$ provided
$\varepsilon<\rho(1-\beta)/C_1.$
The proof is complete.
\end{proof}
The next lemma deals with oscillation reduction on two nested meshes.
\begin{lemma}\label{oscreductionthm}
Let $\Th$ be a conforming refinement of $\mathcal{T}_H$ and $\mathcal{R}_H$ be the set of refined elements in $\mathcal{T}_H$. Then
$$\osc_{\Th}^2(f,\Th)\leq\osc_{\mathcal{T}_H}^2(f,\mathcal{T}_H)-\lambda_{*}\osc_{\mathcal{T}_H}^2(f,\mathcal{R}_H),$$
where $\lambda_*=1-2^{-\frac{2}{n}}$.
\end{lemma}
\begin{proof}
For $T\in\mathcal{R}_H$, let $\hat{\mathcal{T}}_h:=\{t\in\Th: t\subset T\}.$ Then
\begin{equation*}
\begin{aligned}
&\sum_{t\in\hat{\mathcal{T}}_h}h_{t}^2\|f-f_{\Th}\|^2_{t}=\sum_{t\in\hat{\mathcal{T}}_h}|t|^{\frac{2}{n}}\|f-f_{\Th}\|^2_{t}\\
&\quad=2^{-\frac{2}{n}}h_T^2\sum_{t\in\hat{\mathcal{T}}_h}\|f-f_{\Th}\|^2_{t}\leq2^{-\frac{2}{n}}h_T^2\|f-f_{\mathcal{T}_H}\|^2_{T},
\end{aligned}
\end{equation*}
which implies
\begin{equation*}
\sum_{t\subset T,\  T\in\mathcal{R}_H}\osc^2_{\Th}(f,t)\leq2^{-\frac{2}{n}}\osc_{\mathcal{T}_H}^2(f,\mathcal{R}_H).
\end{equation*}
and thus
\begin{equation}\label{osc1}
\sum_{t\subset T,\  T\in\mathcal{R}_H}\osc^2_{\Th}(f,T)+\lambda_*\osc_{\mathcal{T}_H}^2(f,\mathcal{R}_H)\leq\osc_{\mathcal{T}_H}^2(f,\mathcal{R}_H).
\end{equation}
For $T\in\mathcal{T}_H\backslash\mathcal{R}_H$,
\begin{equation}\label{osc2}
\osc^2_{\Th}(f,T)=\osc_{\mathcal{T}_H}^2(f,T).
\end{equation}
Combining \eqref{osc1} and \eqref{osc2}, the proof is complete.
\end{proof}

With the above results we next prove contraction of \textsf{AMFEM}.
\begin{theorem}\label{termination}\emph{(contraction)}
Let $\{\sigma_\ell,\mathcal{T}_\ell\}_{\ell\geq0}$ be a sequence of solutions and meshes produced by \textsf{AMFEM}. For any $0<\varepsilon<(1-\beta)/(C_1 C_3)$, there exist $\rho, \zeta>0$ and 0$<\gamma<1$ depending only $\varepsilon, \theta$ and $\mathcal{T}_0$ such that,
\begin{equation*}
\begin{aligned}
&(1-\varepsilon)\| \sigma - \sigma_{\ell+1}\|^2 + \rho\eta_{\mathcal{T}_{\ell+1}}^2(\sigma_{\ell+1},\mathcal{T}_{\ell+1}) + \zeta\osc_{\mathcal{T}_{\ell+1}}^2( f, \mathcal{T}_{\ell+1})\\
&\leq\gamma\left\{(1-\varepsilon)\| \sigma - \sigma_\ell\|^2 + \rho\eta_{\mathcal{T}_{\ell}}^2(\sigma_\ell, \mathcal{T}_\ell) + \zeta\osc_{\mathcal{T}_{\ell}}^2( f, \mathcal{T}_\ell)\right\}.
\end{aligned}
\end{equation*}
\end{theorem}
\begin{proof}
Let $E_\ell=(1-\varepsilon)\| \sigma - \sigma_\ell\|^2 + \rho\eta_{\mathcal{T}_\ell}^2(\sigma_\ell, \mathcal{T}_\ell).$ Theorem \ref{errorreduction} and Lemma \ref{oscreductionthm} give
\begin{align}
E_{\ell+1}&\leq\alpha E_\ell+C_4\hat{o}_\ell,\label{reduction1}\\
o_{\ell+1}&\leq o_\ell-\lambda_*\hat{o}_\ell,\quad0<\lambda_*<1.\label{oscreduction}
\end{align}
Let $\zeta=\lambda_*^{-1}C_4$. Combined \eqref{reduction1}, \eqref{oscreduction}, and 
$$\rho o_\ell\leq\rho\eta_{\ell}\leq E_\ell,$$ 
we have
\begin{equation}\label{reduction2}
\begin{aligned}
E_{\ell+1}+\zeta o_{\ell+1}&\leq\alpha E_\ell+\zeta o_\ell\\
&\leq(\alpha+\zeta\alpha_1\rho^{-1})E_\ell+\zeta(1-\alpha_1)o_\ell\\
&=\gamma\left(E_\ell+\frac{\zeta(1-\alpha_1)}{\alpha+\zeta\alpha_1\rho^{-1}}o_\ell\right),
\end{aligned}
\end{equation}
where $\gamma:=\alpha+\zeta\alpha_1\rho^{-1}$, $\alpha_1\in(0,1)$ is a constant to be determined.
By requiring 
\begin{equation}\label{cond}
\alpha+\alpha_1\zeta\rho^{-1}<1,\quad\frac{1-\alpha_1}{\alpha+\zeta\alpha_1\rho^{-1}}\leq1,
\end{equation}
\eqref{reduction2} implies
$$E_{\ell+1}+\zeta o_{\ell+1}\leq\gamma(E_{\ell}+\zeta o_{\ell}).$$
\eqref{cond} is satisfied by selecting 
$$\frac{\rho(1-\alpha)}{\rho+\zeta}\leq\alpha_1<\min\left(1,\frac{\rho(1-\alpha)}{\zeta}\right).$$
The proof is complete.
\end{proof}

The methods used above to prove convergence have some similarities to prior work.  Our treatment of oscillation, however, uses properties of $\hat{o}_\ell$ that create distinct implementation and efficiency improvements. To clarify this point, next we compare our convergence proof with \cite{CHX} and \cite{YL}. 

In \cite{CHX}, oscillation is not included in the error indicator and therefore there is no control on $o_\ell$ in their quasi-orthogonality result \eqref{CHXquasi}. To enforce the strict reduction on $o_{\ell+1}\leq\kappa o_\ell$ for some $\kappa<1$, the AMFEM in \cite{CHX} imposed a separate marking for data oscillation. Our convergence analysis shows that the marking for data oscillation is somehow artificial. The convergence of \textsf{AMFEM} can be achieved by a single marking step based on the estimator. This improvement essentially results from the sharper quasi-orthogonality Theorem \ref{quasiortho} with the local data oscillation $\hat{o}_\ell$, which can be canceled using Lemma \ref{oscreductionthm} on the oscillation reduction.

The second author  \cite{YL} considered adaptive methods for the Hoge Laplacian problem \eqref{HL} on the de Rham complex with index $1\leq k\leq n$. Of particular interest here is the case $k=n$ for the mixed formulation of Poisson's equation. In particular, the AMFEM in \cite{YL} is a contraction in the error $\|\sigma-\sigma_h\|^2+\hat{\zeta}\|d(\sigma-\sigma_h)\|^2+\hat{\rho}\hat{\eta}^2_{\Th}(\sigma_h,\Th)$, which is generically of lower order than the total error in Theorem \ref{termination}. Since \cite{YL} considered the error $\|\sigma-\sigma_h\|_{H\Lambda^{k-1}}$ in the $V$-norm instead of the $L^2$-norm, an elementary quasi-orthogonality (Lemma 4.1 in \cite{YL})
$$\|\sigma-\sigma_h\|^2\leq\frac{1}{1-\varepsilon}\|\sigma-\sigma_H\|^2-\|\sigma_h-\sigma_H\|^2+\frac{\varepsilon}{1-\varepsilon}\|f_{\Th}-f_{\mathcal{T}_H}\|^2$$
is enough for convergence analysis there.
 
\subsection{ Optimality of \textsf{AMFEM}}
The next theorem  is devoted to a discrete upper bound, which is a common ingredient of optimality proofs in the literature. Similar bound for the Hodge Laplacian problem has already been established in \cite{YL} by using Demlow's technique in \cite{Demlow}. Since the estimator $\eta_{\Th}$ is different from the one when $k=n$ in \cite{YL}, we sketch the proof here.
\begin{theorem} \label{disupper}
\emph{(discrete upper bound)}
Let $\mathcal{T}_h$ be a conforming refinement of $\mathcal{T}_H$ and $\mathcal{R}_H$ be the set of refined elements. There exists  $\widetilde{\mathcal{R}}_H\supset\mathcal{R}_H$, which is the union of $\mathcal{R}_H$ and a collection of neighboring simplices of $\mathcal{R}_H$ with $\#\widetilde{\mathcal{R}}_H-\#\mathcal{R}_H\leq C$, such that
\begin{equation*} 
\| \sigma_h - \sigma_H \|\leq C_5\eta_{\mathcal{T}_H}(\sigma_H,\widetilde{\mathcal{R}}_H). 
\end{equation*}
\end{theorem}

\begin{proof}
The proof requires similar   ingredients needed to prove the continuous upper bound.  
We first perform the discrete Hodge decomposition of $  \sigma_h - \sigma_H$.
\begin{equation*}
\begin{aligned}
\sigma_h-\sigma_H&=P_{\mathfrak{B}_h}(\sigma_h-\sigma_H)+P_{\mathfrak{H}_h}(\sigma_h-\sigma_H)+P_{\mathfrak{Z}_h^\perp}(\sigma_h-\sigma_H)\\
&=(\sigma_h-P_{\mathfrak{Z}_h^\perp}\sigma_H)-P_{\mathfrak{B}_h}\sigma_H-P_{\mathfrak{H}_h}\sigma_H.
\end{aligned}
\end{equation*}
Then each component can be estimated by the same procedure in the proof of continuous upper bound. With minimal modifications in the proofs of Lemmas \ref{part1}, \ref{part2}, and \ref{part3}, we have
\begin{subequations}\label{dispart}
\begin{align}
\|\sigma_h-P_{\mathfrak{Z}_h^\perp}\sigma_H\|&\leq C\osc_{\Th}(f,\mathcal{R}_H),\\
\|P_{\mathfrak{B}_h}\sigma_{H}\|&= \ab{- \sigma_H}{d (\varphi - \Pi_h\varphi)}\label{dBh},\\
\|P_{\mathfrak{H}_h}\sigma_H\|&\leq C_{\mathcal{T}_0}\|\sigma_h-\sigma_H\|,\quad C_{\mathcal{T}_0}<1.
\end{align}
\end{subequations}
To obtain the localized bound 
\begin{equation}\label{dB2}\|P_{\mathfrak{B}_h}\sigma_H\|\leq C\big(\sum_{T\in\widetilde{\mathcal{R}}_H}h_T^2\|\delta\sigma_H\|^2+h_T\|\lr{\tr\star\sigma_H}\|_{\partial T}^2\big)^{\frac{1}{2}},
\end{equation}
we start from $\eqref{dBh}$ and using equations (4.11)-(4.17) in \cite{Demlow}. In the end, the discrete upper bound is proved by following the proof of Theorem \ref{upper} and using \eqref{dispart} and \eqref{dB2}.
\end{proof}

Let  $\mathbb{T}_N=\{\mathcal{T}\text{ is a conforming  refinement of } \mathcal{T}_0: \#\mathcal{T}-\#\mathcal{T}_0\leq N\}$. 
For $s>0$, we define the approximation classes
\begin{align*}
    \mathcal{A}_{s}&:=\{\tau\in H\Lambda^{n-1}(\Omega): |\tau|_s:=\sup_{N>0}\big(N^{s}\inf_{\mathcal{T}\in\mathbb{T}_N}\inf_{\tau_{\mathcal{T}}\in V^n(\mathcal{T})}\|\tau-\tau_{\mathcal{T}}\|\big)<\infty\},\\
    \mathcal{A}^o_{s}&:=\{g\in L^2\Lambda^{n}(\Omega): |g|^o_{s}:=\sup_{N>0}\big(N^{s}\inf_{\mathcal{T}\in\mathbb{T}_N}\osc_{\mathcal{T}}(g,\mathcal{T})\big)<\infty\}.
\end{align*}
To prove the quasi-optimality, an extra module \textsf{APPROX} in \cite{CHX} was assumed. Here we do not use \textsf{APPROX}. However, as in the classical AFEM literature \cite{RS07,CKNS}, we make the following assumptions.
\begin{assumption}\label{assump}
\
\begin{enumerate}
\item The marking parameter $\theta\in(0,\theta_*)$, where
$\theta_*^2=\min(1,\frac{C_2}{C_5})$.
\item The marking step marks a subset $\mathcal{M}_\ell$ with minimal cardinality. 
\item The accumulative cardinality of marked triangles satisfies \eqref{NVB}.
\end{enumerate}
\end{assumption}

The threshold $\theta_*$ for marking parameter $\theta$ comes from the next lemma.
\begin{lemma}\label{optmarking}
\emph{(optimal marking)}
Let $\mathcal{T}$ be a conforming refinement of $\mathcal{T}_0$ and $\sigma_{\mathcal{T}}\in V^n(\mathcal{T})$ be the solution of \eqref{HLDDRC} on $\mathcal{T}$. Set $\mu=1-\frac{\theta^2}{\theta^2_*}$. Let  $\mathcal{T}_*$ be a conforming refinement of $\mathcal{T}$, such that the finite element solution $\sigma_{\mathcal{T}_*}\in V^n(\mathcal{T}_*)$ satisfies
\begin{equation}\label{testreduction}\|\sigma-\sigma_{\mathcal{T}_*}\|^2+\osc^2_{\mathcal{T}_*}(f,\mathcal{T}_*)\leq\mu\left\{\|\sigma-\sigma_{\mathcal{T}}\|^2+\osc^2_{\mathcal{T}}(f,\mathcal{T})\right\}.
\end{equation}
Then the set of enlarged refined elements $\widetilde{\mathcal{R}}$ in Theorem \ref{disupper} verifies the D\"orfler marking property
$$\eta_{\mathcal{T}}(\sigma_\mathcal{T},\widetilde{\mathcal{R}})\geq\theta\eta_{\mathcal{T}}(\sigma_\mathcal{T},\mathcal{T}).$$
\end{lemma} 
\begin{proof}
By Theorem \ref{lower} and \eqref{testreduction},
\begin{equation}\label{interopt}
\begin{aligned}
&(1-\mu)C_2\eta^2_{\mathcal{T}}(\sigma_\mathcal{T},\mathcal{T})\leq(1-\mu)\big(\|\sigma-\sigma_{\mathcal{T}}\|^2+\osc^2_\mathcal{T}(f,\mathcal{T})\big)\\
&\qquad\leq\|\sigma-\sigma_{\mathcal{T}}\|^2-\|\sigma-\sigma_{\mathcal{T}_*}\|^2+\osc^2_\mathcal{T}(f,\mathcal{T})-\osc^2_{\mathcal{T}_*}(f,\mathcal{T}_*)\\
&\qquad\leq\|\sigma_\mathcal{T}-\sigma_{\mathcal{T}_*}\|^2.
\end{aligned}
\end{equation}
In the last step, we use the triangle inequality and $\osc_{\mathcal{T}_*}(f,\mathcal{T}_*)\leq\osc_\mathcal{T}(f,\mathcal{T})$. Then by \eqref{interopt} and Theorem \ref{disupper},
$$\eta^2_{\mathcal{T}}(\sigma_\mathcal{T},\widetilde{\mathcal{R}})\geq\frac{(1-\mu)C_2}{C_5}\eta^2_{\mathcal{T}}(\sigma_\mathcal{T},\mathcal{T}).$$
The proof is complete by $\theta^2_*\leq C_2/C_5$.
\end{proof}
 
 Combining the optimal marking Lemma \ref{optmarking} ,the contraction Theorem \ref{termination}, and the lower bound Theorem \ref{lower}, the quasi-optimality of \textsf{AMFEM} follows from the same proof in \cite{CKNS}, see Lemma 5.10 and Theorem 5.11 there for details. 
\begin{theorem}\emph{(quasi-optimality)}
Let Assumption \ref{assump} be satisfied. If $\sigma\in\mathcal{A}_s$ and $f \in \mathcal{A}_s^{o}$, then there exists $C_6$ depending only on $\theta, s,$ and $\mathcal{T}_0$, such that
\begin{equation*}
\left\{\| \sigma - \sigma_N\|^2+\osc_{\mathcal{T}_N}^2(\sigma_N,\mathcal{T}_N)\right\}^{\frac{1}{2}}\leq C_6( \|\sigma\|_{\mathcal{A}_s} + \|f\|_{\mathcal{A}^{o}_s})( \# \mathcal{T}_N- \# \mathcal{T}_0 )^{-s}.
\end{equation*}
\end{theorem}




\bibliographystyle{abbrv}
\bibliography{zbib/adam,zbib/mjh-papers,zbib/mjh-extra,zbib/HoSt2010b,zbib/library}



\end{document}